\documentclass[review]{article}

\usepackage{lineno,hyperref}
\usepackage{amssymb, amsthm, amsmath}
\usepackage{mathtools}
\usepackage[margin=0.9in]{geometry}
\usepackage{fancyhdr}

\theoremstyle{plain}
\newtheorem{corollary}{Corollary}[section]
\newtheorem{theorem}{Theorem}[section]
\newtheorem{lemma}{Lemma}[section]

\theoremstyle{definition}
\newtheorem{definition}{Definition}[section]

\DeclarePairedDelimiter\floor{\lfloor}{\rfloor}

\title{The Integer-Magic Spectra of Trees}
\author{Alvaro Carbonero$^a$ and Dylan Obata$^b$ \\ $^{a, b}$University of Nevada, Las Vegas \\ $^a$carboa1@unlv.nevada.edu \ \ \ \ $^b$obatad1@unlv.nevada.edu}
\date{\today}

\begin{document}
\setlength\parindent{24pt}

\maketitle

\begin{abstract}
For any positive integer $h$, a graph $G=(V,E)$ is said to be $h$-magic if there exists a labeling $l:E(G)\to \mathbb{Z}_h -\{0\}  $ such that the induced vertex set labeling $\ l^+ : V(G) \to \mathbb{Z}_h \ $ defined by
$$
l^+ (v)=\sum_{uv \in E(G)} \ l(uv)
$$
is a constant map. The integer-magic spectrum of a graph $G$, denoted by $IM(G)$, is the set of all $h \in \mathbb{N} $ for which $G$ is $h$-magic. So far, only the integer-magic spectra of trees of diameter at most five have been determined. In this paper, we determine the integer-magic spectra of trees of diameter six and higher.
\end{abstract}



\section{Introduction}

In this paper all graphs are connected, finite, simple, and undirected. For graph theory notations and terminology not described in this paper, we refer the readers to \cite{kn:c-z}. Let $G=(V, E)$ be a graph. For an abelian group $A, $ written additively, any mapping $\ l: E(G)\to A-\{ 0 \}  $ is called a \textit{labeling}. Given an edge labeling $l$, we can define a labeling of the vertices $l^+:V(G) \rightarrow A$ such that
$$
l^+(v) = \sum_{uv\in E(G)} l(uv);
$$
that is, $l^+(v)$ is the sum of the labels of all the edges incident with $v$. If there is a labeling $l$ such that $l^+$ is constant, then we say $G$ is an $A$-magic graph and we call $l$ an $A$-magic labeling. If $G$ is not $A$-magic for every abelian group $A$, then we say $G$ is \textit{non-magic}. On the other hand, if $G$ is $A$-magic for every abelian group $A$, we say that $G$ is \textit{fully magic}.

For convenience and to follow convention, we will use the notation $h$-magic, where $h>1$, to indicate $\mathbb{Z}_h$-magic. Similarly, if $l$ is a $\mathbb{Z}_h$-magic labeling, we say $l$ is an $h$-magic labeling. It is pertinent to mention that there is another definition of an $h$-magic graph \cite{other}; this one is a different generalization of magic graphs and so it does not relate to the topics covered in this paper.

\begin{definition}
For a graph G, let the integer-magic spectrum of $G$, denoted by $IM(G)$, be the set of all integers $h$ for which $G$ is $h$-magic.
\end{definition}

The integer-magic spectra of trees of diameter at most five have already been determined \cite{diam4, kn:l-si-s}. In this paper, we generalize these previous results and characterize the integer-magic spectra of trees of arbitrary diameter. For a compilation of results regarding the integer-magic spectra of other families of graphs, we refer the reader to \cite{kn:gallian}.

The term \textit{pendant vertex} will refer to vertices with degree one. Similarly, the term \textit{pendant edge} will refer to edges which are incident to pendant vertices. Note that a pendant vertex will have the same label as the pendant edge incident to it, so in order for $l^+$ to be constant, all pendent edges must have the same label. We will reserve the variable $x$ for the label that every pendant edge must have.  Consequently, for $l$ to be an $h$-magic labeling, we need $l^+(v) \equiv x \mod{h}$ for every vertex $v$.

\begin{figure}[!h]
\begin{center}
\includegraphics[scale=0.7]{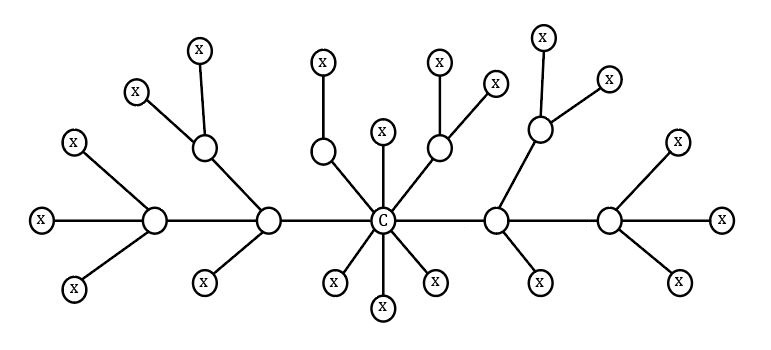}
\end{center}
\caption{A typical tree of diameter $6$.} 
\end{figure}

\section{Notation for Trees}

To determine the integer-magic spectrum of any tree, we need to develop notation that applies to trees of any diameter. To achieve this, we use the fact that trees of even diameter have as a center a single vertex, and that trees of odd diameter have as a center the path $P_2$. From now on, $c$ will denote the vertex at the center of a tree of even diameter; similarly, $c_1$ and $c_2$ will denote the vertices at the center of a tree of odd diameter. Also, to simplify the statement of our theorems, we will use $T_n$ to refer to a tree of diameter $n$. 

We start by describing vertices in terms of their distance from the center.

\begin{definition}
For a tree $T_{2k}$ and an integer $i\geq 0$, let 
$$
D_i := \{v \in V(T_{2k}) \text{ : } d(v, c) = i\}.
$$
\end{definition}

Trees of diameter $2k+1$ have the added complication of having two vertices at the center. For the sake of having formal definitions, we need to distinguish those vertices that are closer to $c_1$ from those that are closer to $c_2$. With this objective, we define the function $\delta: V(T_{2k+1}) \rightarrow \mathbb{N} \cup \{0 \}$ where $\delta (v) = \text{min} \{ d(v, c_1), d(v, c_2)\}$. 

\begin{definition}
For a tree $T_{2k+1}$ and an integer $i \geq 0$, let 
$$
D_i := \{v \in V(T_{2k+1}) \text{ : } \delta (v) = i\}.
$$
\end{definition}

We are abusing notation by using $D_i$ for both even and odd diameter trees. However, for our purposes, their definitions are equivalent. In both cases, $D_i$ is the collection of vertices at distance $i$ from the center. See Figure 2 for an example.

\begin{figure}[!h]
\begin{center}
\includegraphics[scale=0.55]{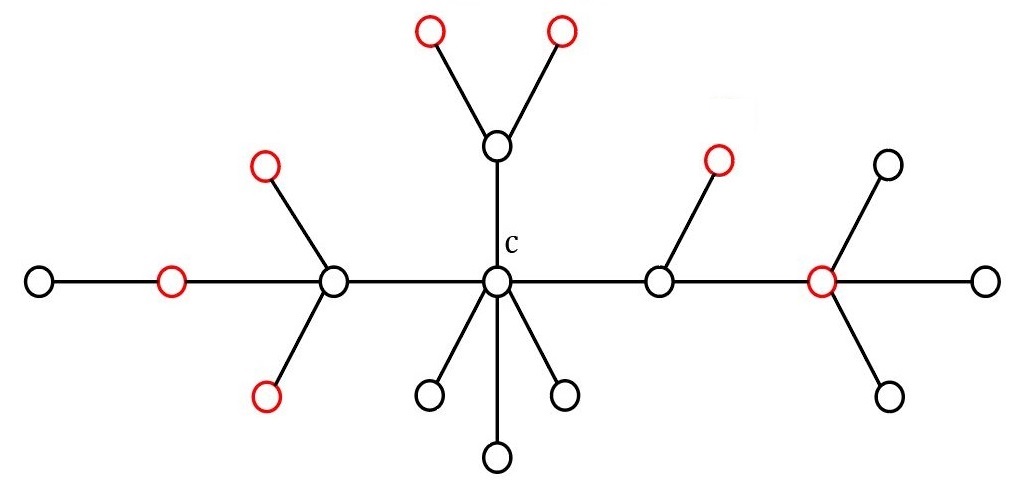}

\end{center}
\caption{A tree $T_6$ with the members of $D_2$ highlighted in red.} 
\end{figure}

Note that for both $2k$ and $2k+1$ diameters, $D_d = \emptyset$ if $d>k$; otherwise, the diameter would be bigger. Furthermore, if the tree has diameter $2k$, then by definition there must be at least two pendant vertices that are distance $k$ away from $c$; similarly, for $T_{2k+1}$, $c_1$ and $c_2$ each must have at least one pendant vertex that is $k$ away. Thus, for both odd and even diameters, if $1 \leq i \leq k$, then $|D_i| \geq 2$. Additionally, even diameters have $D_0 = \{c\}$, and odd diameters have $D_0 = \{c_1, c_2\}.$ Finally, we note that, for any $n$, if $uv \in E(T_n)$ and $u\in D_m$, then either $v\in D_{m+1}$ or $v\in D_{m-1}$ because $u$ and $v$ are adjacent, so the difference in their distance from the center is exactly one.

To discuss $h$-magic labelings for trees, we need a way to talk about the subgraph that ``branches out'' of a vertex. With this objective in mind, we define the branches of $T_n$. To simplify the statement of these definitions, we will use the term $P(u, v)$, where $u, v\in V(T_n)$, to refer to the vertex set of the unique path in $T_n$ from $u$ to $v$.

\begin{definition}
For a tree $T_{2k}$ and $u \in V(T_{2k})\setminus \{c\}$, let the \textit{branch of} $u$, or $B^u$, be the induced subgraph of $T_{2k}$ with vertex set: 
$$
\{v \in V(T_{2k}) \text{ : }d(c, u) \leq d(c, v) \text{ and } \forall w \in P(u, v) \text{ } d(c, u) \leq d(c, w) \}.
$$
\noindent Furthermore, define $B^c = T_{2k}$.
\end{definition}

Informally, $B^u$ is the subgraph of $T_{2k}$ that extends from $u$ and outwards to the pendant vertices. We now make an equivalent definition for trees of odd diameter. 

\begin{definition}
For a tree of $T_{2k+1}$ and $u \in V(T_{2k + 1})\setminus\{c_1, c_2\}$, let the \textit{branch of} $u$, or $B^u$, be the induced subgraph of $T_{2k+1}$ with vertex set:
$$
\{v \in V(T_{2k+1}) \text{ : } \delta (u) \leq \delta(v) \text{ and } \forall w \in P(u, v) \text{ } \delta (u) \leq \delta (w)\}.
$$
\noindent Furthermore, let $B^{c_1}$ be the induced subgraph with vertex set:
$$
\{v \in V(T_{2k+1}) \text{ : } \delta (c_1) \leq \delta(v) \text{ and } c_2 \not\in P(v, c_1)\}.
$$
The definition of $B^{c_2}$ follows similarly.
\end{definition}

\begin{figure}[!h]
\begin{center}
\includegraphics[scale=0.5]{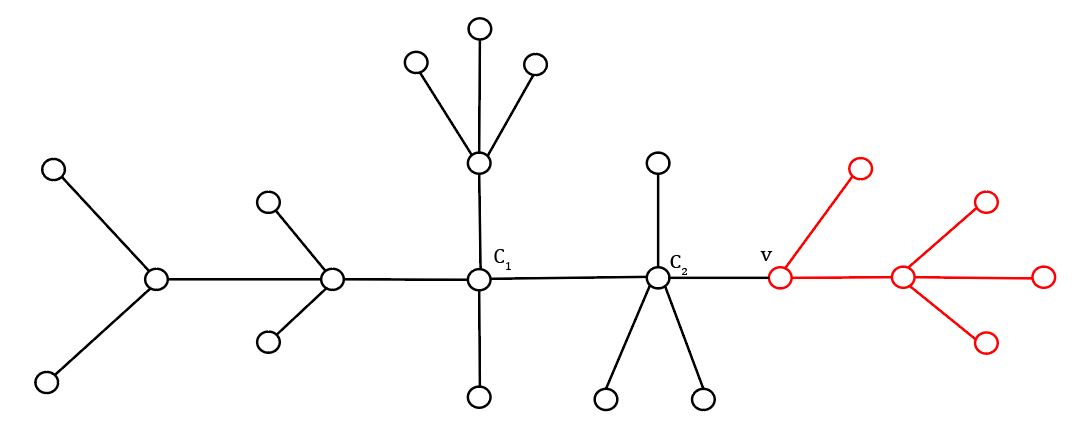}

\end{center}
\caption{A tree $T_7$ with $B^v$ highlighted in red.} 
\end{figure}

If $v \in V(B^u)$ and $v \not = u$, then we say $v$ \textit{extends} from $u$. An important observation to make about these branches is that if $uv\in E(T_n)$, with $u\in D_m$ and $v\in D_{m+1}$, then $B^v$ is a subgraph of $B^u$ since all the vertices that extend from $v$ also extend from $u$. In general, the branch of any vertex that extends from $u$ is a subgraph of $B^u$. We now create a term that combines the objectives behind $D_m$ and $B^u$. 

\begin{definition}
For a tree $T_n$, a vertex $u \in V(T_n)$ and an integer $m\geq 0$, let 
$$
D^u_m := D_m \cap V(B^u).
$$
\end{definition}

Informally, $D^u_m$ is the collection of vertices at distance $i$ from the center that also extend from $u$. An important observation is that if $u\in D_m$ and $u$ is not a pendant vertex, then all the vertices in $D^u_{m+1}$ are adjacent to $u$ because they extend from $u$ and are distance one away. On the other hand, if $u$ is a pendant vertex, then it must be that $D^u_{m+1} = \emptyset$ because nothing would be able to extend from $u$. Finally, if $u\in D_m$, then $D_m^u = \{u\}$ because all other vertices in $V(B^u)$ must be further away from the center than $u$. 

Having the same notation for both even and odd diameter trees allows us to make reference to trees in general without having to specify the parity of its diameter; when the parity of the diameter of the tree is important, we will separate the cases. We now prove a result that makes no reference to $h$-magic labelings; it is a general fact about trees. A special case of the result will be extensively used in the next section.

\begin{theorem}
Let $T_n$ be a tree, let $p$ and $m$ be integers where $p > m$, and let $u$ be a vertex in $D_m$ such that the set $D^u_p$ is nonempty. If $D^u_p = \{w_1, w_2, ..., w_{\alpha} \}$, then, for every integer $d$ such that $d>p$,
\begin{equation} \label{genthm}
\sum_{i=1}^{\alpha}|D^{w_i}_d| = |D^u_d|.    
\end{equation}

\end{theorem}
\begin{proof}
Set $k = \floor{\frac{n}{2}}$. If $d>k$, the result is trivial since all the sets in (\ref{genthm}) are empty. Thus, assume that $d \leq k$.

Observe that the sets $D^{w_1}_d, D^{w_2}_d, ...,$ and $D^{w_\alpha}_d$ are all pairwise disjoint since their corresponding branches in $T_n$ do not overlap. Thus, 
$$
|\bigcup_{1\leq i \leq \alpha} D_d^{w_i}| = \sum_{i=1}^{\alpha}|D^{w_i}_d|.
$$
We finish the proof by showing that $\bigcup_{1\leq i \leq \alpha} D_d^{w_i} = D^u_d$. Let $v\in D_d^{w_i}$ for some $i$ where $1\leq i \leq \alpha$. By definition, $v\in D_d$ and $v\in V(B^{w_i})$. Since $w_i$ extends from $u$, $B^{w_i}$ is a subgraph of $B^u$. Thus, $v \in V(B^u)$. It follows that $v \in D_d^{u}$. Now let $v \in D^u_d$. Then $v \in V(B^u)$. Since $v$ is further away from the center than any of the vertices in $D^u_p$, the path from $u$ to $v$ must go through $w_i$ for some $i$. That is, $v$ extends from $w_i$, and so $v \in V(B^{w_i})$. Since $v \in D_d$, $v \in D^{w_i}_d$. 
\end{proof}

In the next section, we will use the special case of $p=m+1$ since, as discussed before, $D^u_{m+1}$ consists of vertices adjacent to $u$. For convenience and future reference, we state it explicitly.

\begin{corollary} \label{lem1}
For a tree $T_n$, let $u$ be a vertex in $D_m$ such that $u$ is not a pendant vertex. If $D^u_{m+1} = \{w_1, w_2, ..., w_{\alpha}\}$, then, for every $d$ such that $d>m+1$,
$$
\sum_{i=1}^{\alpha}|D^{w_i}_d| = |D^u_d|.
$$
\end{corollary}

\section{The Forced Labeling $f$}

In an $h$-magic labeling, every pendant edge is forced to have the same label. We will show that in trees this notion of a ``forced'' labeling extends beyond pendant edges. 

We start by showing a method for labeling edges that makes $l^+$ constant with the exception of the vertices at the center. The method will depend on the labeling assigned to the pendant edges, though it will not guarantee that the edge labeling does not include 0 in its range.

In Figure 4, we see such method applied. We can label the edges closer to the center according to the label of the edges further away in such a way that their common vertex gets the constant label. For instance, for $v$, the sum of the labels of the edges further away is $-3x$. Thus, the label of the edge closer to the center needs to be $4x$ so that $l^+(v)=-3x+4x = x$. By using the notation we developed in the previous section, we can define a function that calculates the number that multiplies the $x$ for each edge.

\begin{figure}[!h]
\begin{center}
\includegraphics[scale=0.6]{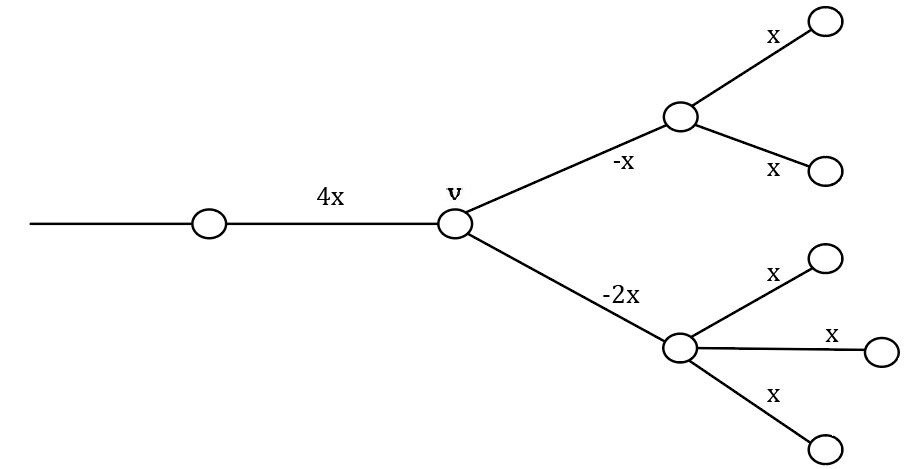}
\end{center}
\caption{Edges labelled such that $l^+$ is constant.} 
\end{figure}

\begin{definition}
For a tree $T_{2k}$, define $f: E(T_{2k}) \rightarrow \mathbb{Z}$ such that if $uv \in E(T_{2k})$, $u\in D_m$ and $v \in D_{m-1}$, then
$$
f(uv) := (-1)^{k-m+1}\sum_{i = 1}^{k-m+1}(-1)^i|D^{u}_{k-i+1}|.
$$
Furthermore, for a tree $T_{2k+1}$, define $f: E(T_{2k+1}) \rightarrow \mathbb{Z}$ such that if $uv \in E(T_{2k+1})\setminus \{ c_1c_2\}$, $u\in D_m$ and $v\in D_{m-1}$, then
$$
f(uv) := (-1)^{k-m+1}\sum_{i = 1}^{k-m+1}(-1)^i|D^{u}_{k-i+1}|.
$$
As for $c_1c_2$, let
$$
f(c_1c_2) := (-1)^{k+1}\sum_{i=1}^{k+1}(-1)^i|D^{c_1}_{k-i+1}|.
$$
\end{definition}

The definition of $f$ for odd diameter trees is more complicated because there is an extra edge at the center to consider. For edges other than that one, $f$ is the same for even and odd diameters. 

Note that, for even diameters, $f$ cannot guarantee that $l^+(c)=x$. A similar problem happens for odd diameters. Although there is an edge at the center, this edge can only guarantee the constant label for $c_1$ or $c_2$, but not always both. We arbitrarily chose to secure the constant label for $c_1$. Another problem with this method is the possibility that $f(e) = 0$ for some edge $e$. If that happens, then this method cannot be applied to construct an $h$-magic labeling since $0$ cannot be assigned to an edge. These problems, however, do not happen in $h$-magic graphs. The proof of this fact follows

\begin{lemma}\label{mainlem}
Let $T_{n}$ be an $h$-magic tree where $l$ is an $h$-magic labeling. If the pendant edges have label $x$, then for every edge $e$, 
$$
l(e) \equiv xf(e) \mod{h}.
$$
\end{lemma}
\begin{proof}
We first prove the statement for $n=2k$. Define
$$
\hat{D}_m = \{ uv \in E(T_{2k}) \text{ : } u \in D_m \text{ and } v \in D_{m-1}\}
$$
for $m$ such that $1 \leq m \leq k$. This partitions $E(T_{2k})$ into $k$ sets. We will prove the statement by using the principle of backwards induction on $m$.

For the base case, let $uv \in \hat{D}_k$ where $u \in D_k$ and $v \in D_{k-1}$. Since $u\in D_k$, then $uv$ is a pendant edge. By definition, $l(uv) = x$, and so
\begin{align*}
    xf(uv) & = x(-1)^{k-k+1} \sum_{i = 1}^{k-k+1}(-1)^i|D^u_{k-i+1}|\\
           & = x|D^u_{k}| \\
           & = x\\
           & = l(uv) \mod{h}.
\end{align*}
We used the fact that for all $u\in D_i$, $|D^u_i| = 1$.

Now assume that the statement is true for all edges in $\hat{D}_m$. Let $uv\in \hat{D}_{m-1}$, where $u\in D_{m-1}$, so $v\in D_{m-2}$. Assume $u$ is a pendant vertex. By definition, $l(uv) = x$, and so
\begin{align*}
    xf(uv) & = x(-1)^{k-m+1+1}\sum_{i=1}^{k-m+1+1}(-1)^i|D^u_{k-i+1}| \\
           & = x(-1)^{k-m}(-|D^u_k|+...+(-1)^{k-m+2}|D^u_{m-1}|) \\
           & = x(-1)^{k-m}(0 + (-1)^{k-m}) \\
           & = x \\
           & = l(uv) \mod{h}.
\end{align*}
We used the fact that, since $u$ is a pendant vertex, $|D^u_i|=0$ for all $i>m-1$.

Assume then that $u$ is not a pendant vertex. Let $D^u_m = \{w_1, w_2, ..., w_{\alpha} \}$. Since $l$ is an $h$-magic labeling,
$$
l^+(u) = l(uv) + \sum_{i=1}^{\alpha}l(uw_i) \equiv x \mod{h}.
$$
But $uw_i \in \hat{D}_m$ for each $1 \leq i \leq \alpha$, so we can apply the induction hypothesis and obtain
\begin{align*}
    \sum_{i=1}^{\alpha}l(uw_i) & \equiv \sum_{i=1}^{\alpha} xf(uw_i) \\
    & = \sum_{i=1}^{\alpha}x(-1)^{k-m+1}\sum_{j=1}^{k-m+1}(-1)^j|D^{w_i}_{k-j+1}| \\
    & = x(-1)^{k-m+1}\sum_{i=1}^{\alpha}\sum_{j=1}^{k-m+1}(-1)^j|D^{w_i}_{k-j+1}| \\
    & = x(-1)^{k-m+1}\sum_{j=1}^{k-m+1}(-1)^j|D^u_{k-j+1}| \mod{h}.
\end{align*}
Note that we used Corollary (\ref{lem1}) in the last step. In the future, we will use it without reference. From the above fact it follows that
\begin{align*}
    l(uv) & \equiv x - \sum_{i=1}^{\alpha}l(uw_i) \\
          & \equiv x - x(-1)^{k-m+1}\sum_{j=1}^{k-m+1}(-1)^j|D^u_{k-j+1}| \\
          & = x(1 + (-1)^{k-m+2}\sum_{j=1}^{k-m+1}(-1)^j|D^u_{k-j+1}|) \\
          & = x((-1)^{2(k-m+2)}|D^u_{m-1}|+(-1)^{k-m+2}\sum_{j=1}^{k-m+1}(-1)^j|D^u_{k-j+1}|) \\
          & = x(-1)^{k-m+2} ((-1)^{k-m+2}|D^u_{m-1}|+\sum_{j=1}^{k-m+1}(-1)^j|D^u_{k-j+1}|) \\
          & = x(-1)^{k-m+2}\sum_{j=1}^{k-m+2}(-1)^j|D^u_{k-j+1}|)\\
          & = xf(uv) \mod{h}.
\end{align*}
Thus, by the principal of backwards induction, the statement is true for all $m$ such that $1 \leq m \leq k$. And so, the statement is true for all edges in trees of even diameter. 

Note that this proof does not depend on the fact that the edges and vertices are in a tree of even diameter. Since $f$ is the same for both odd and even diameters with the exception of $c_1c_2$, the same proof suffices for all edges in an odd diameter tree except for $c_1c_2$, which we now prove separately.

Let $D^{c_1}_1 = \{w_1, w_2, ..., w_{\alpha} \}$. Since $l$ is an $h$-magic labeling, 
$$
    l^+(c_1) = l(c_1c_2) + \sum_{i=1}^{\alpha}l(c_1w_i) \equiv x \mod{h},
$$
and since $c_1w_i \in \hat{D}_1$ for every $1\leq i \leq \alpha$, then $l(c_1w_i) \equiv xf(c_1w_i) \mod{h}$. Thus, 
\begin{align*}
    \sum_{i=1}^{\alpha}l(c_1w_i) & \equiv \sum_{i=1}^{\alpha} xf(c_1w_i) \\
    & =  \sum_{i=1}^{\alpha}x(-1)^k\sum_{j=1}^{k}(-1)^j|D^{w_i}_{k-j+1}| \\
                                 & = x(-1)^{k}\sum_{j=1}^k(-1)^j|D^{c_1}_{k-j+1}| \mod{h}.
\end{align*}
From this, it follows that
\begin{align*}
    l(c_1c_2) & \equiv x - \sum_{i=1}^{\alpha}l(c_1w_i) \\
              & \equiv x(1 + (-1)^{k+1}\sum_{j=1}^k(-1)^j|D^{c_1}_{k-j+1}|) \\
              & = x((-1)^{2(k+1)}|D^{c_1}_0| + (-1)^{k+1}\sum_{j=1}^k(-1)^j|D^{c_1}_{k-j+1}|) \\
              & = x(-1)^{k+1}((-1)^{k+1}|D^{c_1}_0| + \sum_{j=1}^k(-1)^j|D^{c_1}_{k-j+1}|) \\
              & = x(-1)^{k+1}\sum_{j=1}^{k+1}(-1)^j|D^{c_1}_{k-j+1}| \\
              & = xf(c_1c_2) \mod{h }.
\end{align*}

Thus, the lemma is true for all edges of an odd diameter tree, and so the proof is done.
\end{proof}

This result is significant because it allows us to describe any $h$-magic labeling of a tree $T_n$. In particular, it tells us that any such labeling is dictated entirely by $x$, the label of the pendant edges. Thus, determining if a tree $T_n$ is $h$-magic becomes a matter of determining if we can find an $x$ that satisfies the two conditions that $f$ itself does not guarantee. These conditions, again, are that no edge has the label 0 and that $l^+$ is also constant for the center vertices. It turns out that being able to find such an $x$ characterizes $h$-magic graphs.

\begin{theorem} \label{main}
$T_n$ is $h$-magic if and only if there exists an integer $x$ with the following two properties. 

\begin{enumerate}
    \item For every edge $e\in E(T_n)$, 
    \begin{equation}
    xf(e) \not\equiv 0 \mod{h}.    \label{eqn1}
    \end{equation}
    \item  \begin{enumerate}
            \item If $n$ is even, then
            \begin{equation}
            x\sum_{i=1}^{k}(-1)^{k+i}|D_{k-i+1}|   \equiv x \mod{h}    \label{eqn2}
            \end{equation}
            \item If $n$ is odd, then 
            \begin{equation}
            x\sum_{i=1}^{k}(-1)^i|D^{c_1}_{k-i+1}| \equiv x\sum_{i=1}^{k}(-1)^i|D^{c_2}_{k-i+1}| \mod{h}.    \label{eqn3}
            \end{equation}
           \end{enumerate}
\end{enumerate}
\end{theorem}

Before the proof, we note that the first property corresponds to not allowing the edges to have the label 0, and that the second property corresponds to having the center vertices constant under $l^+$. In particular, the left hand side of (\ref{eqn2}) will correspond to $l^+(c)$. Similarly, the left hand side of (\ref{eqn3}) will correspond to $l^+(c_1)$ and the right side to $l^+(c_2)$. Since $f$ by definition guarantees that $l^+(c_1) = x$, equation (\ref{eqn3}) will suffice to show that $l^+$ is constant.

\begin{proof}
Assume $T_n$ is $h$-magic and let $l$ be an $h$-magic labeling. We will demonstrate that if $x$ is the label of the pendant edges under $l$, then $x$ has both properties. We prove that property 1 holds by contradiction. Assume there exists an edge $e$ such that $xf(e) \equiv 0 \mod{h}$. By Lemma \ref{mainlem}, $l(e) \equiv 0 \mod{h}$ too, which contradicts the fact that $l$ is an $h$-magic labeling. Thus, $x$ has property 1.

For property 2, we start with even diameters. Assume $n=2k$ and set $D_1 = \{w_1, w_2, ..., w_{\alpha}\}$. By using Lemma \ref{mainlem}, we get the result:
\begin{eqnarray*}
    x & \equiv & l^+(c) \\
      & = & \sum_{i=1}^{\alpha}l(cw_i) \\
      & \equiv & \sum_{i=1}^{\alpha}xf(cw_i) \\
      & = & x \sum_{i=1}^{\alpha}(-1)^k\sum_{j=1}^k(-1)^j|D^{w_i}_{k-j+1}| \\
      & = & x \sum_{j=1}^{k}(-1)^{k+j}|D_{k-j+1}| \mod{h}.
\end{eqnarray*}

To prove the odd case, we let $n=2k+1$ and set $D^{c_2}_1 = \{w_1, w_2, ..., w_{\alpha} \}$. As before, Lemma \ref{mainlem} gives us the result:

\begin{align*}
    x & \equiv l^+(c_2) \\
      & = l(c_1c_2) + \sum_{i=1}^{\alpha}l(c_2w_i) \\
      & \equiv x(-1)^{k+1}\sum_{i=1}^{k+1}(-1)^i|D^{c_1}_{k-i+1}| + \sum_{i=1}^{\alpha}x(-1)^k\sum_{j=1}^k(-1)^j|D^{w_i}_{k-j+1}| \\
      & = x(-1)^k(-\sum_{i=1}^{k+1}(-1)^i|D^{c_1}_{k-i+1}| + \sum_{j=1}^k(-1)^j|D^{c_2}_{k-j+1}|) \\
      & = x(-1)^k(-\sum_{i=1}^k(-1)^i|D^{c_1}_{k-i+1}| + \sum_{j=1}^k(-1)^j|D^{c_2}_{k-j+1}| + (-1)^k) \\
      & = x(-1)^k(-\sum_{i=1}^k(-1)^i|D^{c_1}_{k-i+1}| + \sum_{j=1}^k(-1)^j|D^{c_2}_{k-j+1}|) + x \mod{h}.
\end{align*}
\noindent This is equivalent to (\ref{eqn3}), so $x$ has property 2, and so the forward direction is done.

For the other direction, we first deal with trees of even diameter. Assume that $n=2k$ and that there exists an integer $x$ such that (\ref{eqn1}) and (\ref{eqn2}) are true. Define $l(e) = xf(e)$. We claim that $l$ is an $h$-magic labeling, i.e. that $l^+(u) \equiv x \mod{h}$ for every $u\in V(T_n)$, so let $u\in V(T_n)$. We split the proof into three cases.

Assume $u\in D_k$. If $v$ is the unique vertex adjacent to $u$, then,
\begin{align*}
    l^+(u) & =  l(uv) \\
          & = xf(uv) \\
          & = -x(-|D^u_k|) \\
          & = x.
\end{align*}

Assume $u\in D_0$, i.e. $u = c$. If we let $D_1^u = \{w_1, w_2, ..., w_{\alpha}\}$, then,
\begin{align*}
    l^+(u) & = \sum_{i=1}^{\alpha}l(uw_i) \\
           & = \sum_{i=1}^{\alpha}xf(uw_i) \\
           & = x \sum_{i=1}^{\alpha}(-1)^k\sum_{j=1}^k(-1)^j|D^{w_i}_{k-j+1}| \\
           & = x \sum_{j=1}^{k}(-1)^{k+j}|D_{k-j+1}| \\
           & \equiv x \mod{h}.
\end{align*}

Finally, assume $u\in D_m$ where $0<m<k$. The proof of Lemma \ref{mainlem} demonstrates that $f(e) = 1$ for every pendant edge. Thus, if $u$ is a pendant vertex, then $l^+(u) = x$. So assume $u$ is not a pendant vertex. Set $v\in D_{m-1}$ such that $uv \in E(T_{2k})$, and $D^u_{m+1} = \{w_1, w_2, ..., w_{\alpha} \}$. We obtain the result from the following.
\begin{align*}
    l^+(u) & = l(uv) + \sum_{i=1}^{\alpha}l(uw_i)\\
           & = x(-1)^{k-m+1}\sum_{i=1}^{k-m+1}(-1)^i|D^u_{k-i+1}| + x\sum_{i=1}^{\alpha}(-1)^{k-m}\sum_{j=1}^{k-m}(-1)^j|D^{w_i}_{k-j+1}| \\
           & = x(-1)^{k-m}(-\sum_{i=1}^{k-m+1}(-1)^i|D^u_{k-i+1}| + \sum_{j=1}^{k-m}(-1)^j|D^u_{k-j+1}|) \\
           & = x(-1)^{k-m}(-(-1)^{k-m+1}|D^u_m|) \\
           & = x.
\end{align*}

This shows that $l$ is an $h$-magic labeling, thus completing the proof for even diameter trees. As for trees with odd  diameter, let $n=2k+1$ and assume there exists an $x$ such that (\ref{eqn1}) and (\ref{eqn3}) hold. As with the even case, define $l(e) = xf(e)$. The proof for $u\in D_i$ for $i>0$ is the same as in the even case. We deal with $D_0= \{c_1, c_2 \}$ separately. Set $D^{c_1}_1= \{w_1, w_2, ..., w_{\alpha} \}$. It follows that
\begin{align*}
    l^+(c_1) & = l(c_1c_2) + \sum_{i=1}^{\alpha}l(c_1w_i) \\
             & = x(-1)^{k+1}\sum_{i=1}^{k+1}(-1)^i|D^{c_1}_{k-i+1}| + \sum_{i=1}^{\alpha}x(-1)^k\sum_{j=1}^{k}(-1)^j|D^{w_i}_{k-j+1}| \\
             & = x(-1)^k(-\sum_{i=1}^{k+1}(-1)^i|D^{c_1}_{k-i+1}|+\sum_{j=1}^{k}(-1)^j|D^{c_1}_{k-j+1}|) \\
             & = x(-1)^k(-(-1)^{k+1}|D^{c_1}_{0}|) \\
             & = x.
\end{align*}

Similarly, if $D^{c_2}_1 = \{w_1, w_2, ..., w_{\alpha} \}$, then

\begin{align*}
    l^+(c_2) & = l(c_1c_2) + \sum_{i=1}^{\alpha}l(c_2w_i) \\
             & = x(-1)^{k+1}\sum_{i=1}^{k+1}(-1)^i|D^{c_1}_{k-i+1}| + \sum_{i=1}^{\alpha}x(-1)^k\sum_{j=1}^{k}(-1)^j|D^{w_i}_{k-j+1}| \\
             & = x(-1)^k(-\sum_{i=1}^{k+1}(-1)^i|D^{c_1}_{k-i+1}|+\sum_{j=1}^{k}(-1)^j|D^{c_2}_{k-j+1}|) \\
             & \equiv x(-1)^k(-\sum_{i=1}^{k+1}(-1)^i|D^{c_1}_{k-i+1}|+\sum_{j=1}^{k}(-1)^j|D^{c_1}_{k-j+1}|) \\
             & = x(-1)^k(-(-1)^{k+1}|D^{c_1}_{0}|) \\
             & = x.
\end{align*}
Thus, $l^+$ is constant for all vertices, and so $T_n$ is $h$-magic. This completes the odd case, and so the proof.
\end{proof}

Having a complete characterization of when $T_n$ is $h$-magic allows us to find the integer-magic spectrum rather easily. In fact, the proof of Salehi \textit{et al.} for trees of diameter 5 in \cite{diam4, kn:l-si-s} gives a template for the proof of our final result. Before going on to the main result, we provide a small and immediate corollary to Theorem \ref{main} that will shed light on the conditions stated in the main result.

\begin{corollary} \label{cor2}
For any $h$-magic tree of diameter $2k$, if $x$ is the label of the pendant edges, then
$$
x\sum_{i=1}^{k+1}(-1)^{k+i}|D_{k-i+1}| \equiv 0 \mod{h}.
$$
Furthermore, for any $h$-magic tree of diameter $2k+1$, if $x$ is the label of the pendant edges, then
$$
x\sum_{i=1}^{k}(-1)^i(|D_{k-i+1}^{c_1}| - |D_{k-i+1}^{c_2}|) \equiv 0 \mod{h}.
$$
\end{corollary}

\section{The Main Result}

The integer-magic spectrum of an arbitrary tree is provided by the following general result.

\begin{theorem} \label{mainmain}
Given a tree $T_n$, let $k = \floor{\frac{n}{2}},$ and let
$$
C := \{m \in \mathbb{Z} \text{ : } m|r \text{ for some } r\in \text{Range}(f)\}.
$$
When $n$ is even, set
$$
\sigma := \sum_{i=1}^{k+1}(-1)^{k+i}|D_{k-i+1}|,
$$
and when $n$ is odd, set
$$
\sigma := \sum_{i=1}^{k}(-1)^i(|D_{k-i+1}^{c_1}| - |D_{k-i+1}^{c_2}|).
$$
If $D$ is the set of all positive divisors of $\sigma$ that are not in $C$, then
\begin{align*}
IM(T_n) = \left\{ \begin{array}{cc}
                \emptyset & \hspace{5mm} \text{if }\sigma \in C; \\
                \mathbb{N}-C & \hspace{5mm} \text{if }\sigma=0; \\
                \bigcup_{d\in D} d\mathbb{N} & \hspace{5mm} \text{otherwise.} \\
                \end{array} \right.
\end{align*}

\end{theorem}
\begin{proof}
We start by observing, from Corollary \ref{cor2}, that if $x$ is the label of the pendant edges in an $h$-magic labeling, then $\sigma x \equiv 0 \mod{h}$. Similarly, any integer $x$ with the property that $\sigma x \equiv 0 \mod{h}$ also has property 2 in Theorem \ref{main}. We now proceed to the cases. 

Assume $\sigma \in C$, and for a contradiction, assume there is an $h$ such that $T_n$ is $h$-magic. Let $l$ be an $h$-magic labeling and $x$ be the label of the pendant edges. We then have that $\sigma x \equiv 0 \mod{h}$. Since $\sigma \in C$, there exist an edge $e$ such that $\sigma | f(e)$, and so $\sigma x | xf(e)$. But then, $xf(e) \equiv 0 \mod{h}$, so $l(e) \equiv 0 \mod{h}$ by Lemma \ref{mainlem}, which is a contradiction. We conclude $IM (T_n) = \emptyset$.

Assume then that $\sigma = 0$. Let $h\in \mathbb{N} - C$. We will prove that $T_n$ is $h$-magic by showing that $x=1$ satisfies both properties in Theorem \ref{main}. From the definition of $C$, we have that $h\nmid f(e)$ for any edge $e$. In other words, $f(e)\not \equiv 0 \mod{h}$, so property 1 is satisfied. Furthermore, $\sigma = 0$ implies property 2. Thus, $T_n$ is $h$-magic, so $h\in IM(T_n).$ For the other direction, assume $h\in IM(T_n)$. Let $l$ be an $h$-magic labeling and $x$ be the label of the pendant edges. Assume for a contradiction that $h|f(e)$ for some edge $e$. It follows that $xf(e) \equiv 0 \mod{h}$, so $l(e) \equiv 0 \mod{h}$ by Lemma \ref{mainlem}, which is a contradiction. We conclude $h \nmid f(e)$ for all edges, so $h\in \mathbb{N}-C$. This proves that $IM(T_n) = \mathbb{N}-C$.

Finally, assume that $\sigma \not\in C$ and $\sigma \not= 0$. We have to prove that $IM(T_n) = \bigcup_{d\in D}d\mathbb{N}$. Assume first that $h\in IM(T_n)$. Let $l$ be an $h$-magic labeling and $x$ be the label of the pendant edges. Set $d= gcd(\sigma, h)$. We claim that $d\in D$. By definition, $d|\sigma$, so we only need to prove that $d\not \in C$. We do so by contradiction, so assume $d|f(e)$ for some $e\in E(T_n)$. From $\sigma x \equiv 0 \mod{h}$, we know $h|\sigma x$. But since $d=gcd(\sigma, h)$, then $h|dx$ too, and since $d|h$, it follows that $\frac{h}{d}|x$. We combine this with $d|f(e)$ to get that $\frac{h}{d}d |xf(e)$, or $xf(e) \equiv 0 \mod{h}$, which is a contradiction with Lemma \ref{mainlem}. Thus, $d\not \in C$, so $d\in D$. And since $d|h$, there exists a number $p$ such that $dp=h$, i.e. $h\in \bigcup_{d\in D}d\mathbb{N}$. 

We prove the other direction. Let $h\in \bigcup_{d\in D}d\mathbb{N}$. By definition, $h$ is a multiple of $d\in D$ so set $h = dq$ where $q \in \mathbb{N}$. We will demonstrate that $q = \frac{h}{d}$ is a choice of $x$ that satisfies both properties in Theorem \ref{main}. From the definition of $D$, $d\not = 1$, so $\frac{h}{d} \not\equiv 0 \mod{h}$. Since $d\not\in C$, $d \nmid f(e)$ for all edges $e$, or $(\frac{h}{d})d \nmid \frac{h}{d}f(e)$. We conclude that $\frac{h}{d}f(e) \not\equiv 0 \mod{h}$ for all edges, so property 1 is satisfied. By definition, $d | \sigma$, so $\frac{\sigma}{d}h \equiv 0 \mod{h}$. Thus, property 2 is satisfied. By Theorem \ref{main}, $T_n$ is $h$-magic, and so $h\in IM(T_n)$. In other words, $IM(T_n) = \bigcup_{d\in D}d\mathbb{N}$, and the proof is done.
\end{proof}




\begin{thebibliography}{99}

\bibitem{kn:c-z} G. Chartrand and P. Zhang, \textit{A First Course in Graph Theory},
 {Dover Publications, New York} (2012).

\bibitem{kn:gallian} J. Gallian, \textit{A Dynamic Survey in Graphs Labeling} (twenty-second edition),
{\it The Electronic Journal of Combinatorics} (2019), 129-132.

\bibitem{diam4} S-M Lee, E. Salehi, and H. Sun, \textit{Integer-Magic Spectra of Trees with Diameter at most Four}, {\it Journal of Combinatorial Mathematics and Combinatorial Computing} {\bf 50} (2004), 3-15.

\bibitem{other} C. Meenakshi, and K.M. Kathiresan, \textit{A Generalization of Magic and Antimagic Labelings of Graphs}, {\it AKCE International Journal of Graphs and Combinatorics} {\bf 16} (2019), 125-144.

\bibitem{kn:l-si-s} E. Salehi and P. Bennett, \textit{Integer-Magic Spectra of Trees of Diameter Five}, {\it Journal of Combinatorial Mathematics and Combinatorial Computing}  {\bf 66} (2008), 105-111.

\end{thebibliography}
\end{document}